\documentclass[12pt]{article}
\usepackage[utf8]{inputenc}
\usepackage{graphicx}

\usepackage{amssymb}
\usepackage{amstext}
\usepackage{amsthm}
\usepackage{mathtools}
\usepackage{enumerate}  

\newtheorem{lemma}{Lemma}

\newtheorem{theorem}{Theorem}
\newtheorem{prop}{Proposition}

\newtheorem{cor}{Corollary}

\newcommand{\BC}{\textsc{bc}}
\DeclareMathOperator{\tr}{tr}

\DeclareMathOperator{\var}{var}

\allowdisplaybreaks

\title{Constraints on Brouwer's Laplacian Spectrum Conjecture}
\author{Joshua N. Cooper}
\date{\today}

\begin{document}

\maketitle

\begin{abstract}
    Brouwer's Conjecture states that, for any graph $G$, the sum of the $k$ largest (combinatorial) Laplacian eigenvalues of $G$ is at most $|E(G)| + \binom{k+1}{2}$, $1 \leq k \leq n$.  We present several interrelated results establishing Brouwer's conjecture $\textsc{bc}_k(G)$ for a wide range of graphs $G$ and parameters $k$.  In particular, we show that (1) $\textsc{bc}_k(G)$ is true for low-arboricity graphs, and in particular for planar $G$ when $k \geq 11$; (2) $\textsc{bc}_k(G)$ is true whenever the variance of the degree sequence is not very high, generalizing previous results for $G$ regular or random; (3) $\textsc{bc}_k(G)$ is true if $G$ belongs to a hereditarily spectrally-bounded class and $k$ is sufficiently large as a function of $k$, in particular $k \geq \sqrt{32n}$ for bipartite graphs; (4) $\textsc{bc}_k(G)$ holds unless $G$ has edge-edit distance $< k \sqrt{2n} = O(n^{3/2})$ from a split graph; (5) no $G$ violates the conjectured upper bound by more than $O(n^{5/4})$, and bipartite $G$ by no more than $O(n)$; and (6) $\textsc{bc}_k(G)$ holds for all $k$ outside an interval of length $O(n^{3/4})$.  Furthermore, we present a surprising negative result: asymptotically almost surely, a uniform random signed complete graph violates the conjectured bound by $\Omega(n)$. 
\end{abstract}

\section{Introduction}

In \cite{BrHa12}, the authors state what has come to known as ``Brouwer's Conjecture.''  For a simple undirected graph $G$ on $n$ vertices, let $A(G)$ denote its adjacency matrix, $D(G)$ the diagonal matrix of its degree sequence $d_1 \geq \cdots \geq d_n$, and $L(G) = D(G) - A(G)$ its (combinatorial) Laplacian matrix.  Write $\{\lambda_i\}_{i=1}^n$ for the multiset of eigenvalues of $L(G)$; we may assume $\lambda_1 \geq \cdots \geq \lambda_n$ because $L(G)$ is real symmetric.  It is well-known that $\lambda_1 \leq n$ and $\lambda_n = 0$.  Brouwer's Conjecture states that, for each $k$, $1 \leq k \leq n$,
\begin{equation} \label{eq:bcstatement}
\sum_{i=1}^k \lambda_i \leq |E(G)| + \binom{k+1}{2},
\end{equation}
a claim we henceforth refer to as $\BC_k(G)$.  

At first, the inequality (\ref{eq:bcstatement}) may seem somewhat mysterious since, for example, it is highly inhomogeneous and, while the left-hand side is a concave function of $k$, the right-hand side is convex.  Therefore, we discuss here some of the motivation for Brouwer's Conjecture and some of the many partial results known about it.

Write $s_k$ for the quantity $\sum_{i=1}^k \lambda_i$.  It is a key observation that $s_k$ is the same as the $k$-th Ky Fan norm $\|L(G)\|_{(k)}$, defined to be the sum of its $k$ largest singular values (see \cite{Bh97}).  We may assume that $G$ is connected by applying the convexity of $\binom{k+1}{2}$ as a function of $k$ and the fact that the spectrum of $L(G_1 \cup G_2)$ is the multiset union of the spectra of $L(G_1)$ and $L(G_2)$ when the $G_i$ are disjoint.  Then $\BC_1(G)$ follows from the fact that $s_1 = \lambda_1 \leq n$, since $|E(G)| + \binom{1+1}{2} \geq n$ for any connected $G$.  Similarly, $\BC_n(G)$ and $\BC_{n-1}(G)$ hold because 
$$
s_{n-1} = s_n = \tr(L(G)) = \sum_{i=1}^n d_i = 2 |E(G)| \leq |E(G)| + \binom{n}{2},
$$
where we are using that $\lambda_n = 0$. Note that -- writing $\overline{G} = (V(G), \binom{V(G)}{2} \setminus E(G))$ for the complement, $J_n$ for the all-ones matrix, and $I_n$ for the identity matrix -- we have 
\begin{align*}
    L(\overline{G}) &= D(\overline{G}) - A(\overline{G}) \\
    & = ((n-1) I_n - D(G)) - (J_n - I_n - A(G)) = n I_n - J_n - L(G).
\end{align*}
Since the eigenspace of $L(G)$ corresponding to $\lambda_n$ is the span of the all-ones vector, which is the sole nonzero eigenspace of $J_n$, this implies that for $1 \leq k \leq n-2$,
\begin{align} 
s_{n-k-1}(\overline{G}) &= n(n-k-1) - 2 |E(G)| + s_{k}(G) \label{eq:complementation} \\
&= |E(\overline{G})| + \binom{n-k}{2} - |E(G)| - \binom{k+1}{2} + s_{k}(G) \nonumber
\end{align}
from which it follows that, if $\BC_k(G)$ holds, then $\BC_{n-k-1}(\overline{G})$ does as well.  

Let $G' = G + K_1$ denote the join of $G$ and a single new vertex $x$, i.e., $G' = (V(G) \cup \{x\},E(G) \cup \{xy : y \in V(G)\})$.  Since $G' = \overline{\overline{G} \cup \{x\}}$, it is easy to see that the spectrum of $G'$ is
$$
n+1, \lambda_{1}+1, \lambda_{2}+1, \ldots , \lambda_{n-1}+1, 0
$$
so that $s_k(G') = n + k + s_{k-1}(G) = |E(G')| - |E(G)| + \binom{k+1}{2} - \binom{k}{2} + s_{k-1}(G)$, from which it follows that, if $\BC_{k-1}(G)$ holds, then $\BC_{k}(G')$ holds.  In particular, the family of all {\em threshold graphs}, defined by a number of equivalent characterizations, including: can be obtained by complementation and disjoint union with a $K_1$, starting from the empty graph; has the form $G = (V,\{xy : f(x)+f(y) \geq B\})$ where $f : V \rightarrow \mathbb{R}$ is any function and $B \in \mathbb{R}$; has no induced $C_4$, $\overline{C_4}$, or $P_4$ subgraph; is both a cograph and a split graph.

That (\ref{eq:bcstatement}) holds for threshold graphs and is tight for this class is perhaps the strongest motivation for Brouwer's Conjecture, since threshold graphs are believed to maximize the quantities $s_k(G)$.  (Another major motivation for the conjecture comes from the closely related Bai's Theorem, previously known as the Grone-Merris Conjecture; see \cite{Ba11} and \cite{BrHa12}.)  Note that the above observations also imply that graphs generated by disjoint unions and complements from graphs known to satisfy the conjecture do as well -- including, in particular, all cographs. 

Other results on Brouwer's Conjecture include that $\BC_k(G)$ holds for any graph with at most $10$ vertices\footnote{In fact, by using the results of the present manuscript to narrow down the computation from 1,006,700,565 connected graphs on 11 vertices, the author has confirmed this result for $n \leq 11$ via \texttt{SageMath} and \texttt{nauty}.}, via an exhaustive computation by Brouwer; $\BC_2(G)$ -- and therefore $\BC_{n-3}(G)$ by (\ref{eq:complementation}) -- always holds (see \cite{HaMoTa10}); trees/forests by \cite{HaMoTa10}; regular graphs and split graphs (whose vertex set can be partitioned into a clique and an independent set) by \cite{Ma10}; unicyclic and bicyclic graphs by \cite{DuZh12}; and random graphs asymptotically almost surely, by \cite{Ro19}.

Below, we further restrict the range of possible counterexamples to the conjecture in the following ways.  Here and throughout, we use $n$ to denote the number of vertices and $m$ to denote the number of edges when it is clear what graph is referred to.
\begin{enumerate}
    \item In Section \ref{sec:basics}, we show that $\BC_k(G)$ is true for $k \geq 4 \Upsilon - 1$ in graphs with arboricity $\Upsilon$ and it is true for $k \geq 2 \Delta + 3$ in graphs with maximum degree $\Delta$.  In particular, $\BC_k(G)$ holds for $k \geq 11$ when $G$ is planar.
    \item Define the {\em maximum subgraph spectral density} $t = t(G) \in [0,1)$ to be $\max_{\emptyset \subsetneq S \subset V(G)} \rho(G[S])/|S|$ where $\rho(G)$ is the (adjacency) spectral radius of $G$ and $G[S]$ denotes the subgraph induced by $S$.  In Section \ref{sec:basics}, we show that $\BC_k(G)$ holds if $k \geq 2m^{1/3}/(1-t)^{2/3}$, $m \geq (2n)^{3/2}/(1-t)$, or $k \geq \sqrt{8n}/(1-t)$.  Since $t \leq 1/2$ for bipartite graphs, in that case this implies $\BC_k(G)$ for $k \geq \sqrt{32n}$ or $m \geq \sqrt{32} n^{3/2}$.
    \item In Section \ref{sec:variance}, we show that the conjecture is true if the variance of its degree sequence is at most $(\beta (1-\beta) n)^2 - \beta/n^2$, where $\beta = 2m/n^2$ is the edge density, implying Brouwer's Conjecture when $\Delta(G) - \delta(G)$ (the difference of maximum and minimum degrees) is at most $2 \overline{d}(n-\overline{d})/n - 1$, where $\overline{d}$ is the average degree.  In particular, this vastly generalizes the aforementioned results for regular and random graphs, since the degree-sequence variance of the former is $0$ and the latter (assuming the Erd\H{o}s-R\'{e}nyi model $G(n,p)$ with $\max\{p,1-p\} = \omega(n^{-1})$) has variance $\approx n \beta (1-\beta)$ with high probability.
    \item The minimum edge-edit distance of a graph $G$ from a split graph is known as its {\em splittance} $\sigma(G)$.  In Section \ref{sec:splitness}, we show that $\BC_k(G)$ holds if $k \leq \sigma(G)/\sqrt{2n}$, so the full conjecture holds if $\sigma(G) \geq n^{3/2}/\sqrt{2}$.  Then we conclude that, in general, $s_k$ exceeds $m + \binom{k+1}{2}$ by at most $(2n)^{3/4} \sqrt{k} = O(n^{5/4})$, and by at most $O(n)$ for bipartite graphs.  Since Mayank (\cite{Ma10}) showed that this excess is at most $\sigma(G)$, and Hammer-Simeone (\cite{HaSi81}) showed that $\sigma(G) \leq (n+1)(n-2)/8$ and that this is tight, our result is a substantial improvement.
    \item In Section \ref{sec:window}, we show that $\BC_k(G)$ holds for all $k$ outside an interval of length $2^{1/4} n^{3/4}$.
\end{enumerate}
Finally, we present a surprising negative result. It has also been conjectured -- and some limited results are known, see \cite{AsOmTa13,ChHaJiLi18} -- that the bound (\ref{eq:bcstatement}) holds for the signless Laplacian $Q(G) = D(G) + A(G)$ as well. Let $G$ be a signed graph (a graph whose edges have been assigned labels $\{\pm 1\}$) and in that case define $A(G)$ to have entries equal to these labels when they are nonzero, and $L(G) = D(G) - A(G)$ (where $D(G)$ is the ordinary diagonal degree matrix).  These matrices interpolate between Laplacian matrices and signless Laplacian matrices, because the former is obtained from graphs all of whose signs are $+1$ and the latter from graphs all of whose signs are $-1$.  Nikiforov (\cite{Ni19}) suggested investigating the conjecture for signed graphs, and indeed discovered a small ($n=5$ and $m=6$) example of a signed graph $G$ for which $L(G)$ violates the bound.  In Section \ref{sec:signed} we show that in fact, asymptotically almost surely, a uniformly random signed complete graph violates the bound by $\Omega(n)$. 

\section{Preliminaries} \label{sec:basics}

Let $v_1,\ldots,v_n$ denote the vertices of $G$, with $d_i = \deg(v_i)$ and $d_1 \geq \ldots \geq d_n$.  Denote $S_k = \{v_1,\ldots,v_k\}$, write $\overline{T}$ for $V(G) \setminus T$ when $T \subset V(G)$, and let $e(T) = |E(G[T])|$.  Then, since $\|D\|_{(k)} = \sum_{i=1}^k d_i = |E(G)| + |E(G[S_k])| - |E(G[\overline{S_k}])|$,
\begin{equation} \label{eq:m1m2energy}
\| L \|_{(k)} \leq \| D \|_{(k)} + \| A \|_{(k)} \leq m + e(S_k) - e(\overline{S_k}) + \|A\|_{(k)}.
\end{equation}
We will repeatedly use this fact, as well as the facts that $\|A\|_{(k)} \leq \sqrt{2km}$ and $\|A\|_{(k)} \leq n (\sqrt{k}+1)/2$, results that appear in \cite{Ni16}.  Let 
$$
t = t(G) :=  \max_{\substack{S \subset V(G) \\ S \neq \emptyset}} \frac{\rho(G[S])}{|S|} < 1
$$
be the maximum subgraph spectral density of $G$.  Below, we repeatedly make use of the well-known fact that $\rho(G) \geq \overline{d}(G)$, the average degree.

Assuming $m \leq k^3(1-t)^2/8$ yields
\begin{align*}
    \| L \|_{(k)} &\leq m + \frac{tk^2}{2} + \sqrt{2km} \\
    &< m + \binom{k+1}{2},
\end{align*}
since $|E(G[S_k])| \leq |S_k| \rho(G[S_k])/2 \leq k(tk)/2 = tk^2/2$.  So, $\BC_k(G)$ holds if $m \leq k^3(1-t)^2/8$, which is equivalent to $k \geq 2m^{1/3}/(1-t)^{2/3}$.  Note that $\BC_k(G)$ also holds if
$$
kn \leq m + \binom{k+1}{2}
$$
because $\|L\|_{(k)} \leq k \mu_1 \leq kn$.  This follows immediately if $k \leq m/n$.  We may assume $m > k^3(1-t)^2/8$, so $k > m/n$ implies $2m^{1/3}/(1-t)^{2/3} > m/n$, i.e., $m < (2n)^{3/2}/(1-t)$.  Thus,

\begin{prop} \label{prop:use-t}
$\BC_k(G)$ holds for any graph with $m$ edges and $n$ vertices if either $k \geq 2m^{1/3}/(1-t)^{2/3}$ or $m \geq (2n)^{3/2}/(1-t)$, where
$$
t = t(G) :=  \max_{\substack{S \subset V(G) \\ S \neq \emptyset}} \frac{\rho(G[S])}{|S|} < 1
$$
is the maximum subgraph spectral density of $G$.
\end{prop}

Thus, any graph $G$ for which $\BC_k(G)$ fails to hold has $k < 2m^{1/3}/(1-t)^{2/3}$, $m < (2n)^{3/2}/(1-t)$, and $k > m/n$, so
$$
\frac{m}{n} < k < \frac{2((2n)^{3/2}/(1-t))^{1/3}}{(1-t)^{2/3}} = \frac{\sqrt{8n}}{1-t}
$$
Recall that the {\em arboricity} $\Upsilon(G)$ is defined the be the smallest $r$ so that $G$ is a union of $r$ forests.

\begin{prop} \label{prop:arboricity}
For a graph $G$ with arboricity $\Upsilon$, $\BC_k(G)$ holds for $k\geq 4\Upsilon - 1$.
\end{prop}
\begin{proof}
Note that, if a graph $G$ has arboricity $\Upsilon$, then decomposing it into forests $T_1,\ldots,T_\Upsilon$ yields, by Theorem 5 of \cite{HaMoTa10},
\begin{align*}
    \|L(G)\|_k &\leq \sum_i^\Upsilon \|L(T_i)\| \\
    &\leq \sum_i^\Upsilon \left [ m_i + (2k-1) \right ]\\
    &= m + \Upsilon (2k-1).
\end{align*}
Thus, if $\Upsilon \leq k(k+1)/(4k-2)$, then the conjectured bound is satisfied.  This holds if $k \geq 4\Upsilon - 1$.  
\end{proof}

For example, planar graphs have arboricity at most $3$, whence $\BC_k(G)$ holds for $k \geq 11$.

\begin{cor}
For a graph $G$ with maximum degree $\Delta$, $bc_k(G)$ holds for any $k \geq 2\Delta + 3$.
\end{cor}
\begin{proof}
Since $\Upsilon \leq \lfloor \Delta/2 \rfloor + 1$, we also have that bc holds for $k \geq 2 \Delta + 3$ by Proposition \ref{prop:arboricity}.
\end{proof}

\section{Trace of the Square} \label{sec:variance}

We may write
\begin{align*}
\tr (L-\lambda_k I)^2 &= \sum_{i=1}^n (\lambda_i - \lambda_k)^2 \\
&= \sum_{i=1}^k (\lambda_i - \lambda_k)^2 + \sum_{i=k+1}^{n} (\lambda_k - \lambda_i)^2 \\
& \geq \frac{\left (\sum_{i=1}^k \lambda_i - \lambda_k \right )^2}{k} + \frac{\left (\sum_{i=k+1}^n \lambda_k - \lambda_i \right )^2}{n-k} \\
&= \frac{(s_k - k \lambda_k)^2}{k} + \frac{((n-k) \lambda_k - 2m + s_k)^2}{n-k}
\end{align*}
On the other hand, writing $D$ for $\sum_v d_v^2$,
\begin{align*}
\tr (L-\lambda_k I)^2 &= D + 2m - 2 \lambda_k \tr L + \lambda_k^2 n \\
&= D + 2m - 4m \lambda_k + \lambda_k^2 n,
\end{align*}
so,
\begin{equation} \label{eq4}
D + 2m - 4m \lambda_k + \lambda_k^2 n \geq \frac{(s_k - k \lambda_k)^2}{k} + \frac{((n-k) \lambda_k - 2m + s_k)^2}{n-k}.
\end{equation}
Simplifying, we obtain
$$
n s_k^2 - 4mks_k + (4km^2 - k(n-k)(D+2m)) \leq 0
$$
and solving for $s_k$, reparametrizing with $\alpha = k/n$, $m = \beta n^2/2$, and $\tau = D/n^3$,
\begin{align*}
\frac{s_k}{n^2}  & \leq \frac{2km}{n^3} + \sqrt{ \frac{4m^2k^2}{n^6} - \frac{4km^2}{n^5} + \frac{k(n-k)(D+2m)}{n^5}} \\
& = \frac{2km}{n^3} + \sqrt{ - \frac{4m^2k(n-k)}{n^6} + \frac{k(n-k)(D+2m)}{n^5}} \\
&  = \alpha \beta + \sqrt{\alpha (1- \alpha)} \sqrt{\tau - \beta^2 + \beta/n } .
\end{align*}
The target upper bound from Brouwer's Conjecture -- at least asymptotically -- is $s_k/n^2 \leq \beta/2 + \alpha^2/2$.  Note that $\tau = D/n^3$ satisfies $\beta^2 \leq \tau$ by Cauchy-Schwarz, so the quantity under the radical is positive.  When the graph is regular, i.e., $\tau = \beta^2$, we have $s_k/n^2 \lesssim \alpha \beta \leq \beta/2 + \alpha^2/2$.
Since 
$$
\frac{\alpha^2 + \beta}{2} - \alpha \beta = \frac{\beta - \beta^2}{2} + \frac{\alpha^2 + \beta^2}{2} - \alpha \beta \geq \frac{\beta(1-\beta)}{2}
$$
then $\tau \leq \beta^2(1 + (1-\beta)^2) - \beta/n$, i.e.,
$$
\sum_v d_v^2 \leq \left (1 - \frac{2m}{n^2} + \frac{2m^2}{n^4} \right ) \frac{8m^2}{n}  - \frac{2m}{n^3}
$$
implies
\begin{align*}
\frac{s_k}{n^2}  & \leq \alpha \beta + \sqrt{\alpha (1- \alpha)} \sqrt{\tau - \beta^2 + \beta/n } \\
& \leq \alpha \beta + \sqrt{\alpha (1- \alpha)} \sqrt{\beta^2 (1+(1-\beta)^2) - \beta/n - \beta^2 + \beta/n } \\
& = \alpha \beta + \sqrt{(1-\beta)^2 \alpha (1- \alpha)} \beta \\
& \leq \frac{\alpha^2 + \beta}{2} - \frac{\beta(1-\beta)}{2} + \frac{(1-\beta) \beta}{2}  \\
& = \frac{\alpha^2 + \beta}{2} = \frac{m}{n^2} + \frac{k^2}{2 n^2} < \frac{1}{n^2} \left ( m + \binom{k+1}{2} \right ).
\end{align*}

Thus,

\begin{lemma} \label{lem1}  $\BC_k(G)$ holds for any graph $G$ with
$$
\sum_v d_v^2 \leq \left (1 - \frac{2m}{n^2} + \frac{2m^2}{n^4} \right ) \frac{8m^2}{n}  - \frac{2m}{n^3}.
$$
\end{lemma}

\noindent Below, when we refer to the {\em variance} of a sequence $\{a_i\}_{i=1}^N$, we mean the variance of the random variable $a_X$, where $X$ takes a uniformly random value in $[N]$.

\begin{theorem} \label{thm1}
 $\BC_k(G)$ holds for any graph $G$ whose degree sequence has variance at most $[\beta(1-\beta) n]^2 - \beta/n^2$, where $\beta = 2m/n^2$ is the edge density.
\end{theorem}
\begin{proof}
The hypothesis yields
\begin{align*} 
\beta^2 (1-\beta)^2 n^2 - \beta/n^2 & \geq \var(\{d_v\})\\
&= \frac{1}{n} \sum_v d_v^2 - \frac{1}{n^2} \left ( \sum_v d_v \right )^2 \\
&= \frac{1}{n} \sum_v d_v^2 - \frac{4m^2}{n^2} = \frac{1}{n} \sum_v d_v^2 - \beta^2 n^2
\end{align*}
so that
\begin{align*}
\sum_v d_v^2& \leq \beta^2 (1-\beta)^2 n^3 + \beta^2 n^3 - \frac{\beta}{n} \\
& = n^3 \left ( 2 - 2\beta + \beta^2 \right ) \beta^2 - \frac{\beta}{n} \\
& = \left ( 1 - \frac{2m}{n^2} + \frac{2m^2}{n^4} \right ) \frac{8m^2}{n} - \frac{2m}{n^3},
\end{align*}
and the result follows by applying Lemma \ref{lem1}.
\end{proof}

Note that the variance of the degree sequence of a graph sampled from the Erd\H{o}s-R\'enyi model $G(n,p)$ is $\leq n \beta (1-\beta) (1+o(1))$ with high probability, and $n \beta (1-\beta) \gg \beta^2(1-\beta)^2n^2-\beta/n$ as long as $p = \omega(n^{-1})$ and $1-p = \omega(n^{-1})$, this implies Rocha's result (\cite{Ro19}) that the Brouwer Conjecture holds for random graphs almost surely.  

\begin{cor}
Denoting the maximum, average, and minimum degrees by $\Delta$, $\overline{d}$, and $\delta$, respectively, 
$$
\Delta - \delta \leq \frac{2 \overline{d} \left ( n - \overline{d} \right )}{n} - 1
$$
for a graph $G$, then Brouwer is true for $G$.
\end{cor}
\begin{proof}
Popoviciu's inequality states that $\var(\{d_v\}) \leq (\Delta-\delta)^2/4$.  Thus, by Theorem \ref{thm1}, since $2 \overline{d} \left ( 1 - \overline{d}/n \right ) - 1 \leq 2 \overline{d} (1- \overline{d}/n ) - 2/n^2 = 2\beta(1-\beta)n - 2/n^2$,
\begin{align*}
\var(\{d_v\}) & \leq \frac{(2\beta(1-\beta)n - 2/n^2)^2}{4} \\
& \leq \beta^2(1-\beta)^2n^2 \cdot \left ( 1 - \frac{1}{\beta(1-\beta)n^3} \right )^2 \\
& \leq \beta^2(1-\beta)^2n^2 \cdot \left ( 1 - \frac{1}{\beta(1-\beta)n^3} \right )
\end{align*}
as long as $\beta(1-\beta)n^3 \geq 1$, which is satisfied if the graph is nonempty.  Continuing,
\begin{align*}
\var(\{d_v\}) & \leq \beta^2(1-\beta)^2n^2 - \frac{\beta^2(1-\beta)^2n^2}{\beta(1-\beta)n^3} \\
& \leq \beta^2(1-\beta)^2n^2 - \beta(1-\beta)/n^2 \leq \beta^2(1-\beta)^2n^2 - \beta/n^2
\end{align*}
since $1-\beta \geq 1/n$, from which the result follows per Theorem \ref{thm1}.
\end{proof}

\begin{cor}
If $G$ belongs to a class of graphs with $\Delta+1 < (2 - \epsilon) \overline{d}$ for any fixed $\epsilon > 0$ and $m = o(n^2)$, then $\BC_k(G)$ holds for all $k$ and all sufficiently large $n$.
\end{cor}
\begin{proof}
\begin{align*}
\frac{2 \overline{d} \left ( n - \overline{d} \right )}{n} - 1 &= 2 \overline{d} \left ( 1 - \frac{2m}{n^2} \right )- 1 \\
&= 2 \overline{d} \left ( 1 - o(1) \right )- 1 \\
&> \overline{d} (2-\epsilon) - 1 > \Delta \geq \Delta - \delta.
\end{align*}
\end{proof}

Thus, for example, any sufficiently large $K_r$-free graph with $\Delta < (2-\epsilon) \overline{d}$ satisfies Brouwer's Conjecture.

\section{Nearly-Split Graphs} \label{sec:splitness}

\begin{prop} \label{prop:almostsplit}
Suppose $G$ violates $\BC_k(G)$.  Then
$$
\left | \binom{S_k}{2} \setminus E(G) \right | \leq k\sqrt{2n} - k
$$
and
$$
\left | \overline{S_k} \cap E(G) \right | \leq k\sqrt{2n} - k.
$$
That is, there is a split graph $G'$ with blocks $S_k$ and $\overline{S_k}$ which differs from $G$ on a set of at most $k \sqrt{8n} - 2k$ edges, and, in particular, the splittance of $G$ satisfies $\sigma(G) < k \sqrt{8n}$.
\end{prop}
\begin{proof}
Suppose $G$ violates $\BC_k(G)$.  Then, by (\ref{eq:m1m2energy}),
$$
m + \binom{k+1}{2} < \|L(G)\|_k \leq m + e(S_k) - e(\overline{S_k}) + \|A(G)\|_{(k)}
$$
and, when combined with $\|A\|_{(k)} \leq \sqrt{2km}$ and $k > m/n$, we obtain
$$
e(S_k) > \binom{k}{2} + k - k\sqrt{2n}
$$
and
$$
e(\overline{S_k}) < k \sqrt{2n} - k.
$$
In particular, $G$ is at most $2 k \sqrt{2n} - 2k <  k \sqrt{8n}$ edges away from being a split graph with bipartition $(S_k,\overline{S_k})$.
\end{proof}

\begin{cor}
$\BC_k(G)$ holds if $k \leq \sigma(G)/\sqrt{8n}$, and so in particular holds for all $k$ if $\sigma(G) \geq \sqrt{2} n^{3/2}$.
\end{cor}

\begin{proof}
This is just an application of Proposition \ref{prop:almostsplit}, with the observation that, by (\ref{eq:complementation}) we can assume that $k \leq n/2$.
\end{proof}

\begin{prop} \label{prop:onlym1}
Fix a $G$ and $k \in [n]$, let $G_1 = (V(G), \binom{S_k}{2} \setminus E(G))$ and $G_2 = (V(G),E(G) \cap \binom{\overline{S_k}}{2})$, and denote $m_i = |E(G_i)|$ for $i=1,2$.  Then
$$
\| L(G) \|_{(k)} \leq m + \binom{k+1}{2} + \min\{m_1 - m_2 + \sqrt{2 km_2}, m_2 - m_1 + \sqrt{2 (n-k)m_1} \}.
$$
\end{prop}
\begin{proof}
Since $L(G) = L(G \cup G_1 \setminus G_2) - L(G_1) + L(G_2)$, by Ky Fan's inequality, 
\begin{align*}
    \|L(G)\|_{(k)} &\leq \|L(G \cup G_1 \setminus G_2)\|_{(k)} + \sum_{i=1}^k \lambda_i(L(G_2)-L(G_1)) \\
    &\leq \|L(G \cup G_1 \setminus G_2)\|_{(k)} + \|L(G_2)\|_{(k)} \\
    &\leq m + m_1 - m_2 + \binom{k+1}{2} + \sqrt{2km_2} \\
    &= m + \binom{k+1}{2} + m_1 + \sqrt{2km_2} - m_2.
\end{align*}
where the second inequality follows because $L(G_2)$ and $L(G_1)$ are positive semidefinite and $V(G_1) \cap V(G_2) = \emptyset$ (so the set of eigenvalues of $L(G_2)-L(G_1)$ is just the union of the nonnegative spectrum of $L(G_2)$ and the nonpositive spectrum of $-L(G_1)$); and the third inequality follows because $G \setminus G_2$ is a split graph with $m-m_2$ edges (applying Mayank's confirmation of Brouwer's Conjecture for split graphs, see \cite{Ma10}).  The same statement with $m_1$ and $m_2$ swapped follows by applying (\ref{eq:complementation}).
\end{proof}

\begin{theorem}
For any $G$,
\begin{align*}
    \|L(G)\|_{(k)} - \left ( m + \binom{k+1}{2} \right ) &\leq 2^{3/4} n^{1/4} \sqrt{k(n-k)} \\
    & \leq 2^{3/4} n^{3/4} \sqrt{k} = O(n^{5/4}).
\end{align*}
If $t = t(G)$, then
$$
    \|L(G)\|_{(k)} \leq m + \binom{k+1}{2} + \frac{2^{3/2}}{\sqrt{1-t}} \cdot n,
$$
and for bipartite $G$,
$$
    \|L(G)\|_{(k)} \leq m + \binom{k+1}{2} + 4 n.
$$
\end{theorem}
\begin{proof}
Let $k' = \min\{k,n-k\}$.  Note that, by Proposition \ref{prop:almostsplit} and (\ref{eq:complementation}), $m_1,m_2 < k' \sqrt{2n}$.  Therefore,
\begin{align*}
\| L(G) \|_{(k)} &\leq m + \binom{k+1}{2} + \sqrt{\max\{2(n-k)k'\sqrt{2n},2kk'\sqrt{2n}\}} \\
&= m + \binom{k+1}{2} + 2^{3/4} n^{1/4} \sqrt{\max\{(n-k)k',kk'\}} \\
&\leq m + \binom{k+1}{2} + 2^{3/4} n^{1/4} \sqrt{k' (n-k')} \\
&= m + \binom{k+1}{2} + 2^{3/4} n^{1/4} \sqrt{k(n-k)} \\
&\leq m + \binom{k+1}{2} + 2^{-1/4} n^{5/4}.
\end{align*}
Then, using Proposition \ref{prop:use-t}, $k \leq \sqrt{8n}/(1-t)$ for any $G$ not satisfying $\BC_k(G)$, with $t=1/2$ for bipartite graphs.
\end{proof}

\section{Narrow window of violations} \label{sec:window}

By Proposition \ref{prop:almostsplit}, $m_1, m_2 \leq k \sqrt{2n}$ if $\BC_k(G)$ fails to hold.  Suppose that $G$ also violates the conjecture at $l > k$.  Then, since $S_l \setminus S_k \subseteq \overline{S_k}$,
\begin{align*}
m + \binom{l+1}{2} & < \| L \|_{(l)} \\
&\leq m + e(S_l) - e(\overline{S_l}) + l \sqrt{2n} \qquad \text{by (\ref{eq:m1m2energy})} \\
&\leq m + e(S_k) + e(S_k,S_l \setminus S_k) + e(S_l \setminus S_k) - e(\overline{S_l}) + l \sqrt{2n}  \\
&\leq m + \binom{l}{2} - \binom{l-k}{2} + e(S_l \setminus S_k) - 0 + l \sqrt{2n} \\
&\leq m + \binom{l}{2}-\binom{l-k}{2} + k \sqrt{2n} + l \sqrt{2n} \qquad \text{by Proposition \ref{prop:almostsplit}} \\
&= m + \binom{l}{2}-\binom{l-k}{2} + (k+l) \sqrt{2n} .
\end{align*}
Thus,
\begin{align*}
(k + l) \sqrt{2n} & > \binom{l+1}{2} - \binom{l}{2} + \binom{l-k}{2} \\
& = \frac{(l-k)^2+(l+k)}{2}
\end{align*}
so $\sqrt{2n} > 2[(l-k)^2+(l-k)]/(k+l) > (l-k)^2/l$, i.e., $l - k < (2n)^{1/4} l^{1/2}$.   Similarly, suppose that $G$ violates Brouwer at $l < k$.  Then, since $\overline{S_l} \setminus \overline{S_k} \subseteq S_k$,
\begin{align*}
m + \binom{l+1}{2} & < \| L \|_{(l)} \\
&\leq m + e(S_l) - e(\overline{S_l}) + l \sqrt{2n} \\
&\leq m + \binom{l}{2} - e(\overline{S_l} \setminus \overline{S_k}) + l \sqrt{2n} \\
&\leq m + \binom{l}{2} - \left (\binom{k-l}{2} - k \sqrt{2n} \right )  + l \sqrt{2n} \\
&= m + \binom{l}{2}-\binom{k-l}{2} + (k + l) \sqrt{2n},
\end{align*}
from which it follows that $k -l < (2n)^{1/4} k^{1/2}$.  Thus, $|l-k| < (2n)^{1/4} \max\{k,l\}^{1/2} < 2^{1/4} n^{1/4 + 1/2}$.

\begin{theorem} \label{prop:smallwindow}
If $\BC_k(G)$ and $\BC_l(G)$ fail to hold, then 
$$
|l-k| < (2n)^{1/4} \max \{k,l\}^{1/4} < 2^{1/4} n^{3/4}
$$
\end{theorem}

\section{Signed Graphs} \label{sec:signed}

Let $G^\tau$ be a {\em signed graph}, i.e., a simple undirected graph $G$ and a function $\tau : E(G) \rightarrow \{-1,1\}$.  Let $A(G) \in \mathbb{R}^{V(G) \times V(G)}$ denote its {\em signed adjacency matrix}, where $A(G)_{v,w} = \tau(\{v,w\})$, let $D(G)$ denote the diagonal degree matrix of $G$, and let $L(G) = D(G) - A(G)$ denote the {\em signed Laplacian matrix} of $G$. Write $m = |E(G)|$ for the number of edges in $G$ and $n = |V(G)|$ for the number of vertices. Note that, for $\tau \equiv 1$ (i.e., ordinary graphs), $L(G^\tau)$ is the ordinary Laplacian of $G$, and, for $\tau \equiv -1$, $L(G^\tau)$ is $Q(G)$, the ``unsigned Laplacian matrix'' of $G$.  We denote the eigenvalues of $L(G^\tau)$, as before, by $\lambda_1 \geq \ldots \geq \lambda_n = 0$.  Somewhat surprisingly, although Brouwer's conjecture states that $s_k = \sum_{j=1}^k \lambda_j$ satisfies $s_k \leq |E(G)| + \binom{k+1}{2}$ for $\tau \equiv 1$, and the same is conjectured (see, for example, \cite{AsOmTa13,ChHaJiLi18}) for $\tau \equiv -1$ (and both are easy to show for $G = K_n$), the statement is false for almost all $\tau$ and $K_n$.

\begin{theorem}
Asymptotically almost surely, for $\tau$ chosen uniformly at random from $\{-1,1\}^{E(G)}$ with $G = K_n$, there exists a $k \in [n]$ so that
$$
s_k := \sum_{j=1}^k \lambda_k > m + \binom{k+1}{2}.
$$
\end{theorem}
\begin{proof}
Note that
$$
L(G^\tau) = n I_n + \sqrt{n} M
$$
where $I_n$ is the $n\times n$ identity matrix, $J$ is the all-ones $n\times n$ matrix, and $M$ is a random symmetric matrix with $0$ diagonal so that its entries $M_{ij}$ above the diagonal are iid random variables taking each of the values $1/\sqrt{n}$ and $-1/\sqrt{n}$ with probability $1/2$.  Then the eigenvalues $\lambda_k$ of $L(G^\tau)$ are in bijection with the eigenvalues $\nu_k$ of $M$, $\nu_n \geq \cdots \geq \nu_1$, via $\lambda_k = n + \sqrt{n} \nu_k$.  Letting $\rho(x) = \frac{1}{2\pi} \sqrt{(4-x^2)_+}$ and $\mu$ the eigenvalue density function $n^{-1} \sum_{i=1}^n \delta_{\nu_i}$, the small-scale semicircle law (see Theorem 2.8 of \cite{BeKn16}) implies that:

\begin{theorem} \label{thm:semicircle} For all $\epsilon, D > 0$,
$$
\Pr \left [ \left | \int_\theta^2 \mu(x) - \rho(x) \, dx \right | > n^{-1+\epsilon} \right ] \leq n^{-D}
$$
for all $n > n_0(D,\epsilon)$.
\end{theorem}

For $k \in [n]$, let $\theta_k$ be defined by
$$
n \int_{\theta_k}^2 \rho(x) \, dx = k - \frac{1}{2}.
$$
Then, the accompanying eigenvalue rigidity theorem (Theorem 2.9 in \cite{BeKn16}) implies that:

\begin{theorem} \label{thm:rigidity} For all $\epsilon, D > 0$, and any $k \in [n]$,
$$
\Pr \left [ \left | \lambda_i - \theta_i \right | > n^{-2/3+\epsilon} \max \{k,n+1-k\}^{-1/3} \right ] \leq n^{-D}
$$
for all $n > n_0(D,\epsilon)$.
\end{theorem}

Thus, fixing $\epsilon, D > 0$ and taking $n_0$ to be the maximum of the two functions in Theorem \ref{thm:semicircle} and \ref{thm:rigidity}, with probability $\geq 2n^{-D}$ we have
\begin{align*}
    \left | \int_{\nu_k}^2 x (\mu(x)-\rho(x)) \, dx \right |
    &= \left | \left [ x \int_{-2}^x \mu(t) - \rho(t) \, dt \right ]_{\nu_k}^2 \!\!\! - \int_{\nu_k}^2 \int_{-2}^x \mu(t) - \rho(t) \, dt \, dx \right | \\    
    &= \left | - \nu_k \int_{-2}^{\nu_k} \mu(t) - \rho(t) \, dt - \int_{\nu_k}^2 \int_{-2}^x \mu(t) - \rho(t) \, dt \, dx \right | \\
    &\leq \nu_k n^{-1+\epsilon} + (2-\nu_k)n^{-1+\epsilon} = 2n^{-1+\epsilon},
\end{align*}
so that, using that $|\rho(x)| \leq 1/\pi$ and writing $\pm C$ for some quantity bounded in absolute value by $C$,
\begin{align*}
\sum_{j=1}^k \nu_j &= n \int_{\nu_k}^2 x \mu(x) \, dx = n \int_{\nu_k}^2 x \rho(x) \, dx \pm 2n^{\epsilon} \\
&= n \int_{\theta_k}^2 x \rho(x) \, dx \pm 2n^{\epsilon} + n \int_{\theta_k}^{\nu_k} x \rho(x) \, dx\\
&= n \int_{\theta_k}^2 x \rho(x) \, dx \pm 2n^{\epsilon} \pm \frac{2}{\pi} n \left | \nu_k - \theta_k \right | \\
&= n \int_{\theta_k}^2 x \rho(x) \, dx \pm 2n^{\epsilon} \pm \frac{2}{\pi} n^{1/3+\epsilon} (n/2)^{-1/3} \\
&= \frac{n}{2\pi} \int_{\theta_k}^2 x \sqrt{4-x^2} \, dx \pm 2n^{\epsilon} \pm \frac{2}{\pi} n^{1/3+\epsilon} n^{-1/3} \\
&> \frac{n (4-f^{-1}(k)^2)^{3/2}}{6 \pi} - 3 n^{\epsilon},
\end{align*}
where $f(x) = \frac{1}{2} + \frac{n}{2\pi} \int_{x}^2 \sqrt{4-t^2} \, dt$.  Thus,
\begin{align}
s_k &= \sum_{j=1}^k \lambda_j = \sum_{j=1}^k \left ( n + \sqrt{n} \nu_j \right ) = kn + \sqrt{n} \sum_{j=1}^k \nu_j \nonumber \\
& > kn + \frac{n^{3/2} (4-f^{-1}(k)^2)^{3/2}}{6 \pi} - 3 n^{1/2+\epsilon} \label{eq:wigner}.
\end{align}
Let $k = n - n^{1/2}+1/2$.  Note that $f$ is a decreasing function on $[-2,2]$, so $f^{-1}$ is decreasing as well. We claim that
$$
\int_{2-q}^2 \sqrt{4-t^2} \, dt \geq \frac{2 \sqrt{2}}{3} q^{3/2}
$$
when $0 \leq q \leq 2$. Indeed, both sides are zero when $q = 0$, and the derivative of the left-hand side satisfies
$$
\frac{d}{dq} \int_{2-q}^2 \sqrt{4-t^2} \, dt = \sqrt{4-(2-q)^2} = \sqrt{4q-q^2} \geq (2 q)^{1/2} = \frac{d}{dq} \left ( \frac{2\sqrt{2}}{3} q^{3/2} \right ).
$$
So, $f^{-1}(n-\sqrt{n}+1/2) < -2+q$ with $q = (9\pi^2/(2n))^{1/3}$ since
\begin{align*}
    f(-2+q) &= \frac{1}{2} + \frac{n}{2\pi} \int_{-2+q}^2 \sqrt{4-t^2} \, dt \\
    &= \frac{1}{2} + n - \frac{n}{2\pi} \int_{2-q}^{2} \sqrt{4-t^2} \, dt \\
    &\leq \frac{1}{2} + n - \frac{n}{2\pi}\cdot \frac{2 \sqrt{2}}{3} q^{3/2}  \\
    &= n - \sqrt{n}+1/2.
\end{align*}
Then
\begin{align*}
    \frac{n^{3/2} (4-f^{-1}(k)^2)^{3/2}}{6 \pi}  &\geq \frac{n^{3/2} \left [ 4-\left (-2+\left (\frac{9\pi^2}{2n} \right )^{1/3} \right )^2 \right ]^{3/2}}{6 \pi} \\
    &= \frac{n^{3/2} \left [ 4 \left ( \frac{9\pi^2}{2n} \right )^{1/3} - \left ( \frac{9\pi^2}{2n} \right )^{2/3} \right ]^{3/2}}{6 \pi} \\
    &\geq \frac{n^{3/2} \left ( 2 \left ( \frac{9\pi^2}{2n} \right )^{1/3} \right )^{3/2}}{6 \pi} = \frac{3n}{2}.
\end{align*}
if $n \geq 6 > 9 \pi^2/16$.  Therefore, by (\ref{eq:wigner}), we have
\begin{align*}
    s_k &> \left ( n - n^{1/2}+ \frac{1}{2} \right ) n + \frac{n^{3/2} (4-f^{-1}(n - n^{1/2} + 1/2)^2)^{3/2}}{6 \pi} - 3 n^{1/2+\epsilon} \\
    &> n^2 - n^{3/2} + \frac{3n}{2} - 3 n^{1/2+\epsilon} \\
    &= \frac{n(n-1)}{2} + \frac{n(n+1)}{2} - n^{3/2} + \frac{3n}{2} - 3 n^{1/2+\epsilon} \\
    &= |E(G)| + \frac{(k+n^{1/2}-1/2)(k+n^{1/2}+1/2)}{2} - n^{3/2} + \frac{3n}{2} - 3 n^{1/2+\epsilon} \\
    &= |E(G)| + \frac{k(k+1) + 2n^{1/2}k - k + n - 1/4}{2} - n^{3/2} + \frac{3n}{2} - 3 n^{1/2+\epsilon} \\
    &> |E(G)| + \binom{k+1}{2} + \frac{2n^{1/2}k}{2} - n^{3/2} - 3 n^{1/2+\epsilon} - \frac{1}{8}  + \frac{3n}{2} \\
    &= |E(G)| + \binom{k+1}{2} + \frac{2n^{1/2}(n-n^{1/2})}{2} - n^{3/2} - 3 n^{1/2+\epsilon}- \frac{1}{8}  + \frac{3n}{2}\\
    &= |E(G)| + \binom{k+1}{2} - n - 3 n^{1/2+\epsilon}- \frac{1}{8} + \frac{3n}{2} \\
    &= |E(G)| + \binom{k+1}{2} + \frac{n}{2} (1+o(1)).
\end{align*}
\end{proof}

\section*{Acknowledgments}

An enormous thank you to V.~Nikiforov for numerous enlightening conversations and insightful questions, and to the Department of Mathematics at the University of Memphis for hosting the author during the period when this work was carried out.  Thanks as well to An Chang, Lei Zhang, and Yirong Zheng for helpful discussions and pointers to relevant literature.

\bibliography{refs}{}
\bibliographystyle{plain}
\end{document}